\documentclass[reqno,11pt]{amsart}
\usepackage{amsfonts}
\usepackage{amsmath}
\usepackage{amssymb}
\usepackage{amscd}
\usepackage{bbm}
\usepackage{amsthm}
\usepackage{amscd}
\usepackage[hidelinks]{hyperref}
\usepackage[initials, shortalphabetic]{amsrefs}

\theoremstyle{definition}
\newtheorem{theorem}{Theorem}[section]
\newtheorem{lemma}[theorem]{Lemma}

\begin{document}
\title[Analogs of Kirchberg's embedding theorem]{Operator space and operator
system analogs of Kirchberg's nuclear embedding theorem}
\author{Martino Lupini}
\address{Martino Lupini\\
Fields Institute for Mathematical Sciences\\
222 College Street\\
Toronto ON M5T 3J1, Canada.}
\email{mlupini@yorku.ca}
\urladdr{http://www.lupini.org/}
\thanks{The author was partially supported by the York University Susan Mann
Dissertation Scholarship. Part of this work was done while the author was a
visitor at the Institut des Hautes \'{E}tudes Scientifiques. The hospitality
of the Institute is gratefully acknowledged.}
\dedicatory{}

\begin{abstract}
The Gurarij operator space $\mathbb{NG}$ introduced by Oikhberg is the
unique separable $1$-exact operator space that is approximately injective in
the category of $1$-exact operator spaces and completely isometric linear
maps. We prove that a separable operator space $X$ is nuclear if and only if
there exist a linear complete isometry $\varphi :X\rightarrow \mathbb{NG}$
and a completely contractive projection from $\mathbb{NG}$ onto the range of 
$\varphi $. This can be seen as the operator space analog of Kirchberg's
nuclear embedding theorem. We also establish the natural operator system
analog of Kirchberg's nuclear embedding theorem involving the Gurarij
operator system $\mathbb{GS}$.
\end{abstract}

\maketitle


\section{Introduction\label{Section:introduction}}

Nuclearity and exactness are properties of fundamental importance for the
theory of C*-algebras and the classification program. The celebrated
Kirchberg exact embedding theorem characterizes (up to *-isomorphism) the
separable exact C*-algebras as the C*-subalgebras of the Cuntz algebra $%
\mathcal{O}_{2}$ \cite[Theorem 6.3.11]{rordam_classification_2002}; see also %
\cites{kirchberg_exact_1995,kirchberg_embedding_2000}. Furthermore the
Kirchberg nuclear embedding theorem asserts that (up to *-isomorphism) the
separable nuclear C*-algebras are precisely the C*-subalgebras of $\mathcal{O%
}_{2}$ that are the range of a (completely) contractive projection \cite[%
Theorem 6.3.12]{rordam_classification_2002}. The Cuntz algebra $\mathcal{O}%
_{2}$, initially introduced and studied in \cite{cuntz_simple_1977}, is the
C*-algebra generated by two isometries of the Hilbert space with
orthogonally complementary ranges.

The main result of this paper is an analog of Kirchberg's nuclear embedding
theorem for operator spaces involving the \emph{(noncommutative) Gurarij
operator space}. This is the unique separable $1$-exact operator space $%
\mathbb{NG}$ which is approximately injective in the category of $1$-exact
operator spaces with completely isometric maps. In other words $\mathbb{NG}$
is characterized by the following property: for any $n\in \mathbb{N}$, $%
\varepsilon >0$, finite-dimensional $1$-exact operator spaces $E\subset F$,
and complete isometry $\phi :E\rightarrow \mathbb{NG}$, there exists a
complete isometry $\psi :F\rightarrow \mathbb{NG}$ such that $\left\Vert
\psi _{|E}-\phi \right\Vert \leq \varepsilon $. Such a space can be thought
as a noncommutative analog of the Gurarij Banach space $\mathbb{G}$, which
is defined by the same property above where one considers Banach spaces
instead of $1$-exact operator spaces %
\cites{gurarij_spaces_1966,lusky_gurarij_1976,kubis_proof_2013}.

Existence of the Gurarij operator space was first proved in \cite%
{oikhberg_non-commutative_2006}, while uniqueness was established in \cite%
{lupini_uniqueness_2014}. It is worth mentioning here that, albeit being $1$%
-exact, $\mathbb{NG}$ does not admit any completely isometric embedding into
an exact C*-algebra by \cite[Corollary 4.16]{lupini_universal_2014}. In
particular $\mathbb{NG}$ is not completely isometric to a C*-algebra.

A separable operator space is $1$-exact if and only if it admits a
completely isometric embedding into $\mathbb{NG}$ \cite[Theorem 4.3]%
{lupini_uniqueness_2014}. (Modulo uniqueness of $\mathbb{NG}$, this also
follows from \cite[Theorem 1.1]{oikhberg_non-commutative_2006} and the proof
of \cite[Theorem 4.7]{effros_injectivity_2001}.) Such a result can be
regarded as an operator space version of Kirchberg's exact embedding
theorem. In this paper we prove the natural analog of Kirchberg's nuclear
embedding theorem for operator spaces: the separable nuclear operator spaces
are (up to completely isometric isomorphism) precisely the subspaces of $%
\mathbb{NG}$ that are the range of a completely contractive projection; see
Theorem \ref{Theorem:embedding-sp}.

We also observe that all the results mentioned above hold for operator
systems, as long as one only considers \emph{unital }linear maps, and
replaces the Gurarij operator space $\mathbb{NG}$ with the Gurarij operator
system $\mathbb{GS}$. Such an operator system is uniquely characterized
among separable $1$-exact operator systems similarly as $\mathbb{NG}$, where
one considers complete order embeddings instead of complete isometries.
Existence and uniqueness of such an operator system, as well as universality
among separable $1$-exact operator systems, have been established in \cite%
{lupini_universal_2014}. The analog of Kirchberg's nuclear embedding theorem
in this context asserts that the separable $1$-exact operator systems are
(up to complete order isomorphism) precisely the subsystems of $\mathbb{GS}$
that are the range of a unital completely positive projection; see Theorem %
\ref{Theorem:embedding-sy}.

It is worth mentioning that it should come as no surprise that, while the
Kirchberg's embedding theorems involve the Cuntz algebra $\mathcal{O}_{2}$,
their operator space and operator system versions involve the Gurarij
operator space and operator system. In fact by \cite[Theorem 1.13 and
Theorem 2.8]{kirchberg_embedding_2000} any separable exact C*-algebra embeds
into $\mathcal{O}_{2}$, and any two unital embeddings of a separable unital
exact C*-algebra into $\mathcal{O}_{2}$ are approximately unitarily
equivalent. It therefore follows from \cite[Theorem 2.21]%
{ben_yaacov_fraisse_2012} that finitely generated exact unital C*-algebras
form a Fra\"{\i}ss\'{e} class with limit $\mathcal{O}_{2}$. The Gurarij
operator space and operator systems can be similarly described as the Fra%
\"{\i}ss\'{e} limits in the sense of \cite{ben_yaacov_fraisse_2012} of the
classes of finitely-generated $1$-exact operator spaces and operator
systems; see \cites{lupini_uniqueness_2014,lupini_universal_2014}. It is
therefore natural to expect that they exhibit similar properties as $%
\mathcal{O}_{2}$ in the respective categories.

The present paper is organized as follows. In Section \ref{Section:spaces}
we prove the above mentioned results concerning operator spaces. We recall a
few basic notions about operator spaces in Subsection \ref{Sub:background-sp}%
. Subsection \ref{Sub:amalgamation-sp} contains canonical approximate
amalgamation results for operator spaces, while a characterization of
nuclearity is recalled in\ Subsection \ref{Sub:nuc}. The operator space
analog of Kirchberg's nuclear embedding theorem is then proved in Subsection %
\ref{Sub:emb-sp}. The analogous results for operator systems are stated in
Section \ref{Section:systems}. The proofs are similar, and only the relevant
changes to be made from the operator space case will be pointed out.

\subsection*{Acknowledgments}

We are grateful to Isaac Goldbring, Ilijas Farah, Timur Oikhberg, and Todor
Tsankov for useful conversations concerning the subject of this article, to
Eusebio Gardella for providing references for Kirchberg's embedding
theorems, and to Ali Kavruk for referring us to \cite{xhabli_super_2012}.

\section{Nuclear operator spaces\label{Section:spaces}}

\subsection{Some background notions on operator spaces\label%
{Sub:background-sp}}

An \emph{operator space} $X$\emph{\ }is a closed subspace of the algebra $%
B\left( H\right) $ of bounded linear operators on a complex Hilbert space $H$%
. The identification of the space $M_{n}\left( X\right) $ of $n\times n$
matrices with entries in $X$ with a space of operators on the $n$-fold
Hilbertian sum $H^{\oplus n}$ induces, for every $n\in \mathbb{N}$, a norm $%
M_{n}\left( X\right) $ (the operator norm). The \emph{matricially normed }%
vector spaces that arise in this way have been characterized by Ruan \cite[%
Theorem 3.1]{ruan_subspaces_1988} as those satisfying the following
condition:%
\begin{equation}
\left\Vert \sum_{i=1}^{m}\alpha _{i}^{\ast }.x_{i}.\beta _{i}\right\Vert
\leq \left\Vert \sum_{i=1}^{m}\alpha _{i}^{\ast }\alpha _{i}\right\Vert
\max_{1\leq i\leq m}\left\Vert x_{i}\right\Vert \left\Vert
\sum_{i=1}^{m}\beta _{i}^{\ast }\beta _{i}\right\Vert \text{\label%
{Equation:op-sp}}
\end{equation}%
where $m\in \mathbb{N}$, $\alpha _{i},\beta _{i}\in M_{k\times n}$ (the
space of complex $k\times n$ matrices), $x_{i}\in M_{k}\left( X\right) $,
the norm of scalar matrices is the operator norm, and $\alpha _{i}^{\ast
}.x_{i}.\beta _{i}$ denotes the row-column matrix multiplication. The
identification of $M_{n}$ with $B\left( H\right) $ where $H$ is the $n$%
-dimensional Hilbert space induces a canonical operator space structure on
the space $M_{n}$ of $n\times n$ complex matrices.

If $\varphi :X\rightarrow Y$ is a linear map between operator spaces, its $n$%
-th\emph{\ amplification} is the map $\varphi ^{\left( n\right)
}:M_{n}\left( X\right) \rightarrow M_{n}\left( Y\right) $ obtained by
applying $\varphi $ entrywise. We set $||\varphi ||_{n}:=||\varphi ^{\left(
n\right) }||$ and $||\varphi ||_{cb}:=\sup_{n}||\varphi ||_{n}$. The map $%
\varphi $ is \emph{completely bounded} if $||\varphi ||_{cb}<+\infty $, 
\emph{completely contractive} if $||\varphi ||_{cb}\leq 1$, and \emph{%
completely isometric} if $\varphi ^{\left( n\right) }$ is isometric for
every $n\in \mathbb{N}$. \emph{In the following all the maps are supposed to
be linear}.

If $X$ is an operator space and $q\in \mathbb{N}$, then $\mathrm{\mathrm{MI}N%
}_{q}\left( X\right) $ is the operator space with same linear structure as $%
X $ and matrix norms%
\begin{equation*}
\left\Vert x\right\Vert _{M_{k}(\mathrm{MIN}_{q}\left( X\right)
)}=\sup_{\phi }\left\Vert \phi ^{\left( k\right) }\left( x\right)
\right\Vert _{M_{k}(M_{q})}
\end{equation*}%
for $x\in M_{k}\left( X\right) $, where $\phi $ ranges over all complete
contractions from $X$ to $M_{q}$. The space $\mathrm{MIN}_{q}\left( X\right) 
$ is characterized by the following property \cite[\S 2]%
{oikhberg_operator_2004}: for any operator space $Z$ and linear map $\varphi
:Z\rightarrow X$, one has that%
\begin{equation*}
\left\Vert \varphi :Z\rightarrow \mathrm{MIN}_{q}\left( X\right) \right\Vert
_{cb}=\left\Vert \varphi :Z\rightarrow X\right\Vert _{q}\text{.}
\end{equation*}%
Following \cite{lehner_mn-espaces_1997}, we say that $X$ is an $M_{q}$\emph{%
-space} if the identity map $id:\mathrm{MIN}_{q}\left( X\right) \rightarrow
X $ is a complete isometry. In an $M_{q}$-space $X$ the norms on $%
M_{k}\left( X\right) $ for $k\neq q$ are completely determined from the norm
in $M_{q}\left( X\right) $. Therefore one can regard an $M_{q}$-space as a
linear space $X$ with a norm on $M_{q}\left( X\right) $. The spaces arising
in this way can be abstractly characterized by the analog of Equation (\ref%
{Equation:op-sp}) where only matrices of size $q$ are considered \cite[Th%
\'{e}or\`{e}me I.1.9]{lehner_mn-espaces_1997}. Furthermore $M_{q}$-spaces
are precisely the operator spaces that admit completely isometric embedding
into $C\left( K,M_{q}\right) $ for some compact Hausdorff space $K$.
Adopting this point of view, \emph{Banach spaces} can be identified with $%
M_{1}$-spaces. The characterizing property of $\mathrm{MIN}_{q}$ shows that
if $\varphi :X\rightarrow Y$ is a linear map between $M_{q}$-spaces, then $%
\left\Vert \varphi \right\Vert _{cb}=\left\Vert \varphi \right\Vert _{q}$.
Furthermore the inclusion functor from the category of $M_{q}$-spaces to the
category of operator spaces determines an equivalence of categories with
inverse the functor $\mathrm{MIN}_{q}$.

A finite-dimensional operator space $X$ is defined to be $1$\emph{-exact }if
for every $\varepsilon >0$ there exist $n\in \mathbb{N}$ and a completely
contractive injective map $\varphi :X\rightarrow M_{n}$ such that $%
\left\Vert \varphi ^{-1}\right\Vert _{cb}\leq 1+\varepsilon $ \cite[\S 17]%
{pisier_introduction_2003}. (Here and in the following, if $\varphi $ is an
injective not necessarily surjective map, we denote by $\varphi ^{-1}$ the
inverse of $\varphi $ regarded as a map from the range of $\varphi $ to the
domain of $\varphi $.) Equivalently $X$ is $1$-exact if and only if for
every $\varepsilon >0$ there exists $q\in \mathbb{N}$ such that%
\begin{equation*}
\left\Vert id_{X}:\mathrm{MIN}_{q}\left( X\right) \rightarrow X\right\Vert
_{cb}\leq 1+\varepsilon \text{.}
\end{equation*}%
An arbitrary operator space $X\ $is then $1$-exact if every
finite-dimensional subspace of $X$ is $1$-exact. Clearly any $M_{q}$-space
is $1$-exact, and the $1$-exact operator spaces are those that can be
locally approximated by $M_{q}$-spaces.

An operator space $X$ is nuclear \cite[\S 14.6]{effros_operator_2000} if for
every finite subset $A$ of $X$ and $\varepsilon >0$ there exist $n\in 
\mathbb{N}$ and completely contractive maps $\varphi :X\rightarrow M_{n}$
and $\psi :M_{n}\rightarrow X$ such that $\left\Vert \left( \psi \circ \phi
\right) \left( x\right) -x\right\Vert \leq \varepsilon $ for every $x\in A$.
A nuclear operator space is in particular $1$-exact \cite[Corollary 14.6.2]%
{effros_operator_2000}. More details about operator spaces can be found in
the monographs %
\cites{effros_operator_2000,pisier_introduction_2003,blecher_operator_2004,paulsen_completely_2002}%
. The notion of $M_{q}$-space has been introduced and studied in \cite%
{lehner_mn-espaces_1997}, and used in 
\cites{oikhberg_operator_2004,oikhberg_non-commutative_2006,
lupini_uniqueness_2014}.

The Gurarij operator space $\mathbb{NG}$ has been defined in \cite%
{oikhberg_non-commutative_2006} to be a separable $1$-exact operator space
satisfying the following approximate injectivity property: whenever $%
E\subset F$ are finite-dimensional $1$-exact operator spaces, $\varphi
:E\rightarrow \mathbb{NG}$ is a complete isometry, and $\varepsilon >0$,
then $\phi $ can be extended to a linear map $\widehat{\phi }:F\rightarrow 
\mathbb{NG}$ such that $||\widehat{\phi }||_{cb}{}||\widehat{\phi }%
^{-1}||_{cb}\leq 1+\varepsilon $. The uniqueness of such a space up to
complete isometry has been proved in \cite{lupini_uniqueness_2014}, where
moreover several equivalent characterization of $\mathbb{NG}$ can be found;
see \cite[Proposition 4.8]{lupini_uniqueness_2014}. It is furthermore proved
therein that a separable operator space is $1$-exact if and only if it
embeds completely isometrically into $\mathbb{NG}$ \cite[Theorem 4.9]%
{lupini_uniqueness_2014}. This can be regarded as an operator space analog
of Kirchberg's exact embedding theorem \cite[Theorem 6.3.11]%
{rordam_classification_2002}. Theorem \ref{Theorem:embedding-sp} below
provides an operator space analog of Kirchberg's nuclear embedding theorem:
a separable operator space $X$ is nuclear if and only if it embeds
completely isometrically as a subspace of $\mathbb{NG}$ that is the range of
a completely contractive projection.

\subsection{Amalgamation of operator spaces\label{Sub:amalgamation-sp}}

The proof of the following lemma is similar to the proof of \cite[Lemma 3.1]%
{lupini_uniqueness_2014}. We present here the main ideas for convenience of
the reader.

\begin{lemma}
\label{Lemma:pushout-sp}Suppose that $X,\widehat{X},Y$ are $M_{q}$-spaces, $%
k\in \mathbb{N}$, and $\delta \geq 0$. Let $\phi :X\rightarrow \widehat{X}$
and $f:X\rightarrow Y$ be linear maps such that $\phi $ is injective, and $%
\left\Vert f\circ \phi ^{-1}\right\Vert _{cb}\leq 1+\delta $. Then there
exist an $M_{q}$-space $\widehat{Y}$, a complete contraction $\widehat{f}:%
\widehat{X}\rightarrow \widehat{Y}$, and a complete isometry $j:Y\rightarrow 
\widehat{Y}$ such that $\left\Vert j\circ f-\widehat{f}\circ \phi
\right\Vert _{k}\leq \delta $. If $f$ is injective and $\left\Vert \phi
\circ f^{-1}\right\Vert _{cb}\leq 1+\delta $, then $\widehat{f}$ is
completely isometric. If $\widehat{X},Y$ are finite-dimensional, then $%
\widehat{Y}$ is finite-dimensional. The space $\widehat{Y}$ has the
following property. If $Z$ is an $M_{q}$-space, and $g:\widehat{X}%
\rightarrow Z$, $i:Y\rightarrow Z$ are completely contractive maps such that 
$\left\Vert i\circ f-g\circ \phi \right\Vert _{k}\leq \delta $, then there
exists a unique linear map $\tau :\widehat{Y}\rightarrow Z$ such that $\tau
\circ \widehat{f}=g$ and $\tau \circ j=i$. Moreover $\tau $ is completely
contractive.
\end{lemma}

\begin{proof}
Let $\widehat{X}\oplus Y$ be the algebraic direct sum of $\widehat{X}$ and $%
Y $. Consider the collection $\mathcal{F}$ of linear maps from $\widehat{X}%
\oplus Y$ to $M_{q}$ of the form%
\begin{equation*}
\left( x,y\right) \mapsto \theta _{X}\left( x\right) +\theta _{Y}\left(
y\right)
\end{equation*}%
where $\theta _{X}:\widehat{X}\rightarrow M_{q}$ and $\theta
_{Y}:Y\rightarrow M_{q}$ are completely contractive maps such that $%
\left\Vert \theta _{X}\circ \phi -\theta _{Y}\circ f\right\Vert _{k}\leq
\delta $. Let $\widehat{Y}$ be the operator space obtained from the vector
space $\widehat{X}\oplus Y$ and the collection of linear maps $\mathcal{F}$
as in \cite[\S 1.2.16]{blecher_operator_2004}. Observe that $\widehat{Y}$ is
in fact an $M_{q}$-space. Let $\widehat{f}:\widehat{X}\rightarrow \widehat{Y}
$ and $j:Y\rightarrow \widehat{Y}$ be the linear maps induced by the first
and second coordinate inclusions. It follows immediately from the
definitions that $\widehat{f}$ and $j$ are completely contractive maps such
that $\left\Vert \widehat{f}\circ \phi -j\circ f\right\Vert _{k}\leq \delta $%
. It remains to show that $j$ is a complete isometry. Suppose that $y\in
M_{q}\left( Y\right) $ and $\eta :Y\rightarrow M_{q}$ is a complete
contraction such that $\left\Vert y\right\Vert =\left\Vert \eta ^{\left(
q\right) }\left( y\right) \right\Vert $. Then $\frac{1}{1+\delta }\left(
\eta \circ f\circ \phi ^{-1}\right) :\phi \left[ X\right] \rightarrow M_{q}$
is a complete contraction and hence it extends to a $\ $complete contraction 
$\psi :\widehat{X}\rightarrow M_{q}$ by injectivity of $M_{q}$ in the
category of $M_{q}$-spaces and complete contractions; see \cite[Proposition
I.1.16]{lehner_mn-espaces_1997}. Observe that $\left\Vert \psi \circ \phi
-\eta \circ f\right\Vert _{q}\leq \delta $. Therefore we have that%
\begin{equation*}
\left\Vert j^{\left( q\right) }\left( y\right) \right\Vert \geq \left\Vert
\eta ^{\left( q\right) }\left( y\right) \right\Vert \geq \left\Vert
y\right\Vert \text{.}
\end{equation*}%
This concludes the proof that $j$ is a complete isometry. The proof that $%
\widehat{f}$ is a complete isometry under the assumption that $f$ is
injective and $\left\Vert \phi \circ f^{-1}\right\Vert _{cb}\leq 1+\delta $
is entirely analogous. Suppose now that $Z$ is an $M_{q}$-space, and $g:%
\widehat{X}\rightarrow Z$ and $i:Y\rightarrow Z$ are completely contractive
maps such that $\left\Vert i\circ f-g\circ \phi \right\Vert _{k}\leq \delta $%
. Then setting%
\begin{equation*}
\tau \left( \left( x,y\right) +N\right) =g\left( x\right) +i\left( y\right)
\end{equation*}%
gives a well defined complete contraction $\tau :\widehat{Y}\rightarrow Z$.
It is not difficult to verify that $\tau $ satisfies the desired conclusions.
\end{proof}

A minor modification of the above argument provides the following
straightforward generalization.

\begin{lemma}
\label{Lemma:multiple-pushout-sp}Suppose that $k,N\in \mathbb{N}$, $X_{n},%
\widehat{X}_{n},Y$ for $n\leq N$ are finite-dimensional $M_{q}$-spaces, $%
k\leq q$, and $\delta _{n}\geq 0$. Suppose that $\phi _{n}:X_{n}\rightarrow 
\widehat{X}_{n}$ and $f_{n}:X_{n}\rightarrow Y$ are linear maps such that 
\begin{equation*}
\max \left\{ \left\Vert f_{n}\circ \phi _{n}^{-1}\right\Vert _{k},\left\Vert
\phi _{n}\circ f_{n}^{-1}\right\Vert _{k}\right\} \leq 1+\delta _{n}
\end{equation*}%
for every $n\leq N$. There exists a finite-dimensional $M_{q}$-space $%
\widehat{Y}$, a complete isometry $j:Y\rightarrow \widehat{Y}$ and complete
isometries $\widehat{f}_{n}:\widehat{X}_{n}\rightarrow Y$ such that $%
\left\Vert \widehat{f}_{n}\circ \phi _{n}-j\circ f_{n}\right\Vert _{k}\leq
\delta _{n}$ for every $n\leq N$. Moreover if $Z$ is an $M_{q}$-space, $%
g_{n}:\widehat{X}_{n}\rightarrow Z$ and $i:Y\rightarrow Z$ are completely
contractive maps such that $\left\Vert g_{n}\circ \phi _{n}-i\circ
f_{n}\right\Vert _{k}\leq \delta _{n}$, then there exists a unique linear
map $\tau :\widehat{Y}\rightarrow Z$ such that $\tau \circ j=i$ and $\tau
\circ \widehat{f}_{n}=g_{n}$. Furthermore $\tau $ is completely contractive.
\end{lemma}

We will refer to the space $\widehat{Y}$ constructed in Lemma \ref%
{Lemma:multiple-pushout-sp} as the $k$\emph{-amalgamated coproduct} over $Y$
of the maps $f_{n}:X_{n}\rightarrow Y_{n}$ and $\phi _{n}:X_{n}\rightarrow Y$
\emph{with tolerance }$\delta _{n}$ for $n\leq N$.

\subsection{A characterization of nuclearity\label{Sub:nuc}}

In the proof of the main theorem we will need the following characterization
of nuclearity for operator spaces.

\begin{lemma}
\label{Lemma:ac-sp}Suppose that $Z$ is a $1$-exact operator space. The
following statements are equivalent:

\begin{enumerate}
\item For every finite-dimensional operator spaces $E,F$, $\delta >0$,
linear maps $\phi :E\rightarrow F$ and $f:E\rightarrow Z$ such that $\phi $
is injective and $\left\Vert f\circ \phi ^{-1}\right\Vert _{cb}<1+\delta $,
there exists a completely contractive map $g:F\rightarrow Z$ such that $%
\left\Vert g_{|E}-f\right\Vert <\delta $;

\item For every $q\in \mathbb{N}$, subspace $E\subset M_{q}$, completely
contractive linear map $f:E\rightarrow Z$, and $\varepsilon >0$, there
exists a completely contractive map $g:M_{q}\rightarrow Z$ such that $%
\left\Vert g\circ \phi -f\right\Vert <\varepsilon $;

\item $Z$ is nuclear.
\end{enumerate}
\end{lemma}

\begin{proof}
The argument is straightfoward. For convenience of the reader, we sketch
below the proofs of the nontrivial implications.

\begin{description}
\item[(2)$\Rightarrow $(3)] Suppose that $\overline{a}$ is a tuple in $Z$
and $\varepsilon >0$. Fix $\delta >0$ small enough. Pick $q\in \mathbb{N}$,
a subspace $E$ of $M_{q}$ and a completely contractive linear map $%
f:E\rightarrow Z$ with range $\mathrm{span}\left( \overline{a}\right) $ such
that $\left\Vert f^{-1}\right\Vert _{cb}<1+\delta $. By assumption there
exits a completely contractive map $\rho :M_{q}\rightarrow Z$ such that $%
\left\Vert \rho -f\right\Vert _{cb}<\delta $. Observe now that $\gamma :=%
\frac{1}{1+\delta }f^{-1}:\mathrm{span}\left( \overline{a}\right)
\rightarrow M_{q}$ is completely contractive and $\left\Vert \rho \circ
\gamma \left( a_{i}\right) -a_{i}\right\Vert <\varepsilon $ for $\delta $
small enough.

\item[(3)$\Rightarrow $(1)] Suppose that $E,F$ are finite-dimensional $1$%
-exact operator spaces, $\delta >0$, and $n\in \mathbb{N}$. Fix linear maps $%
\phi :E\rightarrow F$ and $f:E\rightarrow Z$ such that $\phi $ is injective.
Set $\left\Vert f\circ \phi ^{-1}\right\Vert _{cb}=1+\eta $ and let $%
\varepsilon $ be a strictly positive real number. Since $Z$ is nuclear there
exist $q\in \mathbb{N}$ and completely contractive maps $\rho :Z\rightarrow
M_{q}$ and $\gamma :M_{q}\rightarrow Z$ such that $\left\Vert \gamma \circ
\rho \circ f-f\right\Vert <\varepsilon $. By injectivity of $M_{k}$ there
exists a completely contractive map $h:F\rightarrow M_{k}$ that extends $%
\frac{1}{1+\eta }\left( \rho \circ f\circ \phi ^{-1}\right) $. Let then $g$
be $\gamma \circ h$ and observe that $\left\Vert g\circ \phi -f\right\Vert
<\eta +\varepsilon $.
\end{description}
\end{proof}

\subsection{The nuclear \texorpdfstring{$\mathbb{NG}$}{NG}-embedding theorem \label{Sub:emb-sp}}

The main result of this section is Theorem \ref{Theorem:embedding-sp},
characterizing separable nuclear operator spaces as the ranges of completely
contractive projections of the Gurarij operator space. The proof is inspired
by a characterization of retracts of certain Fra\"{\i}ss\'{e} limits due to\
Dolinka; see \cite[Theorem 3.2]{dolinka_characterization_2012}. Dolinka's
result has later been extended by Kubi\'{s} in \cite{kubis_injective_?} to
more general Fra\"{\i}ss\'{e} limits. The forward implication in Theorem \ref%
{Theorem:embedding-sp} has already been obtained by Oikhberg \cite[Theorem
1.1]{oikhberg_non-commutative_2006} under the stronger assumption that the
operator space $X$ is an $\mathcal{OL}_{\infty ,1+}$ space as defined in 
\cite{effros_scr_1998}.

Fix a sequence $\left( \varepsilon _{n}\right) $ of strictly positive real
numbers such that $\varepsilon _{n}\leq 2^{-2n}$ for every $n\in \mathbb{N}$%
. We say that a subset $D$ of a metric space $X$ is $\varepsilon $-dense for
some $\varepsilon >0$ if every element of $X$ is at distance at most $%
\varepsilon $ from some element of $D$. For $m\in \mathbb{N}$ let $D_{m}$ be
a finite $\varepsilon _{m}$-dense subset of the unit ball of $M_{m}$ and let 
$\left( \overline{a}_{m,i}\right) $ be an enumeration of the finite tuples
from $D_{m}$. Set $E_{m,i}=\left\langle \overline{a}_{m,i}\right\rangle
\subset M_{m}$ for $i,m\in \mathbb{N}$.

Assume that $\left( Z_{n}\right) $ is a direct system of finite-dimensional $%
1$-exact operator spaces with completely isometric connective maps $%
j_{n}:Z_{n}\rightarrow Z_{n+1}$. Suppose that $D_{n}^{Z}\subset \mathrm{Ball}%
\left( Z_{n}\right) $ is an $\varepsilon _{n}$-dense finite subset of $Z_{n}$%
. Define $Z$ to be the corresponding direct limit, and identify $Z_{n}$ with
its image inside $Z$.

\begin{lemma}
\label{Lemma:recognize-NG}Suppose that for any $i,m\in \mathbb{N}$ with $%
i,m\leq n$ and complete contraction $f:E_{m,i}\rightarrow Z_{n}$ such that $%
f\left( \overline{a}_{m,i}\right) \subset D_{n}^{Z}$ there exists a complete
isometry $\widehat{f}:M_{m}\rightarrow Z_{n+1}$ such that $\left\Vert 
\widehat{f}_{|E}-j_{n}\circ f\right\Vert \leq \left\Vert f^{-1}\right\Vert
_{cb}-1+\varepsilon _{n}$. Then $Z$ is completely isometric to $\mathbb{NG}$.
\end{lemma}

\begin{proof}
Observe that $Z$ is separable and $1$-exact. By \cite[Proposition 4.8]%
{lupini_uniqueness_2014} it is enough to show that if $E\subset M_{m}$, $%
g:E\rightarrow Z$ is a complete isometry, and $\delta >0$, then there exists
a complete isometry $\widehat{g}:M_{m}\rightarrow Z$ such that $\left\Vert 
\widehat{g}_{|E}-g\right\Vert \leq \delta $. The \textquotedblleft small
perturbation argument\textquotedblright ---see \cite[Lemma 12.3.15]%
{brown_c*-algebras_2008}---shows that, for such a $g$ and $\delta >0$, there
exist $n\geq i,m$ and a complete contraction $f:E_{m,i}\rightarrow Z_{n}$
such that $\varepsilon _{n}\leq \delta $, $f\left( \overline{a}_{m,i}\right)
\subset D_{n}^{Z}$, $\left\Vert \widehat{f}^{-1}\right\Vert _{cb}\leq
1+\delta $, and $\left\Vert f-g\right\Vert \leq \delta $. The conclusion
then follows easily by applying the hypothesis to $f$.
\end{proof}

\begin{theorem}
\label{Theorem:embedding-sp}Suppose that $X$ is a separable $1$-exact
operator space. Then $X$ is nuclear if and only if there exist a complete
isometry $\phi :X\rightarrow \mathbb{NG}$ and a completely contractive
projection of $\mathbb{NG}$ onto the range of $\phi $.
\end{theorem}

\begin{proof}
It follows easily from the characterizing property of $\mathbb{NG}$ and
Lemma \ref{Lemma:ac-sp} that $\mathbb{NG}$ is nuclear. Therefore the range
of a completely contractive projection of $\mathbb{NG}$ is nuclear as well.
Conversely, suppose that $X$ is a nuclear. Our goal is to construct an
operator space $Z$ completely isometric to $\mathbb{NG}$, a completely
isometric embedding $\eta :X\rightarrow Z$, and a completely contractive map 
$\pi :Z\rightarrow X$ such that $\pi \circ \eta =id_{X}$. We will define by
recursion on $n$

\begin{itemize}
\item an increasing sequence $\left( q_{n}\right) $ in $\mathbb{N}$ such
that $q_{n}\geq n$,

\item an increasing sequence $\left( X_{n}\right) $ of finite-dimensional
subspaces $X$ with inclusion maps $i_{n}:X_{n}\rightarrow X_{n+1}$ such that 
$\bigcup_{n}X_{n}$ is dense in $X$ and 
\begin{equation*}
\left\Vert id_{X_{n}}:\mathrm{MIN}_{q_{n}}\left( X_{n}\right) \rightarrow
X_{n}\right\Vert _{cb}\leq 1+\varepsilon _{n}\text{,}
\end{equation*}

\item an inductive sequence $\left( Z_{n}\right) $ of finite-dimensional $1$%
-exact operator spaces with completely isometric connective maps $%
j_{n}:Z_{n}\rightarrow Z_{n+1}$ such that $Z_{n}$ is an $M_{q_{n}}$-space,

\item completely isometric maps $\eta _{n}:\mathrm{MIN}_{q_{n}}\left(
X_{n}\right) \rightarrow Z_{n}$ such that%
\begin{equation*}
\left\Vert j_{n}\circ \eta _{n}-\eta _{n+1}\circ i_{n}\right\Vert \leq
\varepsilon _{n}\text{,}
\end{equation*}

\item complete contractions $\pi _{n}:Z_{n}\rightarrow X_{n}$ such that 
\begin{equation*}
\left\Vert \pi _{n+1}\circ j_{n}-i_{n}\circ \pi _{n}\right\Vert \leq
\varepsilon _{n}\text{ and }\left\Vert \pi _{n}\circ \eta
_{n}-id_{X_{n}}\right\Vert \leq \varepsilon _{n}\text{,}
\end{equation*}

\item finite $\varepsilon _{n}$-dense subset $D_{n}^{X}$ of $\mathrm{Ball}%
(X_{n})$ and $D_{n}^{Z}$ of $\mathrm{Ball}\left( Z_{n}\right) $ such that $%
i_{n}\left[ D_{n}^{X}\right] \subset D_{n+1}^{X}$, $j_{n}\left[ D_{n}^{Z}%
\right] \subset D_{n+1}^{Z}$, and $\pi _{n}\left[ D_{n}^{Z}\right] \subset
D_{n}^{X}$,
\end{itemize}

such that

\begin{enumerate}
\item For every $i,m\leq n$ and $f:E_{m,i}\rightarrow X_{n}$ complete
contraction such that $f\left( \overline{a}_{m,i}\right) \subset D_{n}^{X}$
there exists a complete contraction $\widehat{f}:M_{m}\rightarrow X_{n+1}$
such that $\left\Vert \widehat{f}_{|E}-i_{n}\circ f\right\Vert \leq
\varepsilon _{n}$;

\item For every $i,m\leq n$ and $f:E_{m,i}\rightarrow Z_{n}$ complete
contraction such that $f\left( \overline{a}_{m,i}\right) \subset D_{n}^{Z}$
there exists a complete isometry $\widehat{f}:M_{m}\rightarrow Z_{n+1}$ such
that $\left\Vert \widehat{f}_{|E}-j_{n}\circ f\right\Vert \leq \left\Vert
f^{-1}\right\Vert _{cb}-1+\varepsilon _{n}$.
\end{enumerate}

Fix a dense sequence $\left( w_{n}\right) $ in $X$. For $k=1$, let $X_{1}=%
\mathrm{span}\left\{ w_{1}\right\} $, $q_{1}$ large enough such that $%
\left\Vert id_{X_{1}}:\mathrm{MIN}_{q_{1}}\left( X_{1}\right) \rightarrow
X_{1}\right\Vert _{cb}\leq 1+\varepsilon _{1}$, and define $Z_{1}=\mathrm{MIN%
}_{q_{1}}\left( X_{1}\right) $. Suppose that $X_{k},Z_{k},\pi _{k},\eta
_{k-1},j_{k-1}$ have been defined for $k\leq n$. Observe that Condition (1)
above concerns only finitely many functions $f$. Therefore one can find a
subspace $X_{n+1}$ of $X$ containing $X_{n}\cup \left\{ w_{n+1}\right\} $
and satisfying Condition (1) by applying finitely many times Proposition \ref%
{Lemma:ac-sp} and the fact that $X$ is nuclear. Pick $q_{n+1}\geq \max
\left\{ n+1,q_{n}\right\} $ such that 
\begin{equation*}
\left\Vert id:\mathrm{MIN}_{q_{n+1}}\left( X_{n+1}\right) \rightarrow
X_{n+1}\right\Vert _{cb}\leq 1+\varepsilon _{n+1}\text{.}
\end{equation*}%
Let $Z_{n+1}$ be the $1$-amalgamated coproduct of $M_{q_{n+1}}$-spaces over $%
Z_{n}$ (see Lemma \ref{Lemma:multiple-pushout-sp}) of the maps $\eta _{n}:%
\mathrm{MIN}_{q_{n}}\left( X_{n}\right) \rightarrow Z_{n}$ and $i_{n}:%
\mathrm{\mathrm{MI}N}_{q_{n}}\left( X_{n}\right) \rightarrow \mathrm{MIN}%
_{q_{n+1}}\left( X_{n}\right) $ with tolerance $\varepsilon _{n}$ and of the
maps $f:E_{m,i}\rightarrow Z_{n}$ and $E_{m,i}\hookrightarrow M_{m}$
(inclusion map) with tolerance $\left\Vert f^{-1}\right\Vert
_{cb}-1+\varepsilon _{n}$, where $m,i$ range in $\left\{ 1,\ldots ,n\right\} 
$ and $f$ range among all the (finitely many) injective complete
contractions such that $f\left( \overline{a}_{m,i}\right) \subset D_{n}^{Z}$%
. Let $j_{n}:Z_{n}\rightarrow Z_{n+1}$ and $\eta _{n+1}:\mathrm{MIN}%
_{q_{n+1}}\left( X_{n+1}\right) \rightarrow Z_{n+1}$ be the canonical
complete isometries such that $\left\Vert j_{n}\circ \eta _{n}-\eta
_{n+1}\circ i_{n}\right\Vert \leq \varepsilon _{n}$. It remains to define
the complete contraction $\pi _{n+1}:Z_{n+1}\rightarrow X_{n+1}$. Observe
that $i_{n}\circ \pi _{n}:Z_{n}\rightarrow X_{n+1}$ is a complete
contraction such that 
\begin{equation*}
\left\Vert i_{n}\circ \pi _{n}\circ \eta _{n}-i_{n}\right\Vert \leq
\varepsilon _{n}\text{.}
\end{equation*}%
Furthermore for every $i,m\leq n$ and for any complete contraction $%
f:E_{m,i}\rightarrow Z_{n}$ such that $f\left( \overline{a}_{m,i}\right)
\subset D_{n}^{Z}$, $\pi _{n}\circ f:E_{m,i}\rightarrow X_{n}$ is a complete
contraction such that $\left( \pi _{n}\circ f\right) \left( \overline{a}%
_{m,i}\right) \subset D_{n}^{X}$. Therefore by Condition (1) there exists a
complete contraction $g:M_{m}\rightarrow X_{n+1}$ such that 
\begin{equation*}
\left\Vert g_{|E_{m,i}}-\left( i_{n}\circ \pi _{n}\right) \circ f\right\Vert
\leq \varepsilon _{n}\text{.}
\end{equation*}%
It follows from the universal property of the $1$-amalgamated coproduct that
there exists a complete contraction $\pi :Z_{n+1}\rightarrow \mathrm{MIN}%
_{q_{n+1}}\left( X_{n+1}\right) $ such that $\eta _{n+1}\circ \pi
=id_{X_{n+1}}$ and $\pi \circ j_{n}=i_{n}\circ \pi _{n}$. Define then $\pi
_{n+1}=\frac{1}{1+\varepsilon _{n+1}}\pi $. This concludes the recursive
construction. We now define $Z$ to be the direct limit $\lim_{\left( \phi
_{n}\right) }Z_{n}$, $\pi :=\lim_{n}\pi _{n}$, and $\eta :=\lim_{n}\eta _{n}$%
. It is immediate to verify that the conditions above ensure that $\pi $ and 
$\eta $ satisfy the desired requirements. Furthermore $Z$ is completely
isometric to $\mathbb{NG}$ by Lemma \ref{Lemma:recognize-NG}.
\end{proof}

\section{Nuclear operator systems \label{Section:systems}}

\subsection{Some background notions about operator systems\label%
{Sub:background-sy}}

An \emph{operator system} is a closed subspace $X$ of $B\left( H\right) $
that contains the identity operator (the \emph{unit}) and it is closed by
taking adjoints. The inclusion $X\subset B\left( H\right) $ and the
identification of $M_{n}\left( X\right) $ with a subspace of $B\left(
H^{\oplus n}\right) $ defines\emph{\ positive cones }on all the matrix
amplifications of $X$. These are induced by the notion of positivity in $%
B\left( H\right) $. (An operator $T\in B\left( H\right) $ is positive
provided that $\left\langle T\xi ,\xi \right\rangle \geq 0$ for any $\xi \in
H$.) The Choi-Effros characterization \cite[Theorem 4.4]%
{choi_injectivity_1977} provides an abstract characterization of operator
systems among the \emph{matricially ordered }vector spaces endowed with a
linear involution (the adjoint map) and a distinguished element (the unit).
An operator system is in particular an operator space, and an abstract
characterization of operator systems among operator spaces with a
distinguished element (the unit) has been given in \cite{blecher_metric_2011}%
.

A map $\varphi $ between operator systems is \emph{unital }if it maps the
unit to the unit, \emph{positive }if it maps positive elements to positive
elements, and \emph{completely positive} if all its amplifications are
positive. A fundamental fact about operator spaces is that the matricial
order structure determines the matricial norms, and vice versa. This implies
that a unital linear map between operator systems is completely positive if
and only if it is completely contractive. A unital completely positive
linear map will be in the following simply called a \emph{ucp map}. A
surjective ucp map with ucp inverse is called a \emph{complete order
isomorphism}. A \emph{complete order embedding }is a complete order
isomorphism onto its image. Observe that a complete order embedding is in
particular a complete isometry.

The minimal operator system structure $\mathrm{OM\mathrm{IN}}_{q}\left(
X\right) $ on an operator system $X$ has been introduced and studied in \cite%
{xhabli_super_2012}. This is defined similarly as $\mathrm{MIN}_{q}\left(
X\right) $ by only considering ucp maps. It is shown in \cite%
{xhabli_super_2012} that $\mathrm{OM\mathrm{IN}}_{q}\left( X\right) $ has
analogous properties as $\mathrm{MIN}_{q}\left( X\right) $ as long as one
replaces (completely) contractive linear maps with (completely) positive
unital linear maps; see \cite[\S 2.7]{lupini_universal_2014}. An $M_{q}$%
\emph{-system} will then be an operator system $X$ such that the identity
map $id:\mathrm{\mathrm{OM\mathrm{IN}}}\left( X\right) \rightarrow X$ is a
complete order isomorphism or, equivalently, $X$ admits a complete order
embedding in $C\left( K,M_{q}\right) $ for some compact Hausdorff space $K$.
The notions of $1$-exact and nuclear operator systems can be defined as for
operator spaces by replacing complete contractions with ucp maps. It turns
out that an operator system is $1$-exact or nuclear, respectively, if and
only if it is $1$-exact or nuclear as an operator space; see \cite[Theorem
3.5]{han_approximation_2011} and \cite[Proposition 5.5]%
{kavruk_quotients_2013}.

\subsection{The nuclear \texorpdfstring{$\mathbb{GS}$}{GS}-embedding theorem \label{Sub:ac-sy}}

The approximate amalgamation results for operator spaces from Subsection \ref%
{Sub:amalgamation-sp} carry over with little change to the operator systems
case. The only extra ingredient needed is \cite[Lemma 3.3]%
{lupini_universal_2014}.

\begin{lemma}
\label{Lemma:multiple-pushout-sy}Suppose that $N\in \mathbb{N}$, $X_{n},%
\widehat{X}_{n},Y$ for $n\leq N$ are finite-dimensional $M_{q}$-systems, $%
k\leq q$, and $\delta _{n}\in \left[ 0,1\right] $ for $n\leq N$. Suppose
that $\phi _{n}:X_{n}\rightarrow \widehat{X}_{n}$ and $f_{n}:X_{n}%
\rightarrow Y$ linear maps such that%
\begin{equation*}
\max \left\{ \left\Vert f_{n}\circ \phi _{n}^{-1}\right\Vert _{k},\left\Vert
\phi _{n}\circ f_{n}^{-1}\right\Vert _{k}\right\} \leq 1+\delta _{n}
\end{equation*}%
for every $n\leq N$. There exists a finite-dimensional $M_{q}$-system $%
\widehat{Y}$, a complete isometry $j:Y\rightarrow \widehat{Y}$ and complete
isometries $\widehat{f}_{n}:\widehat{X}_{n}\rightarrow Y$ such that $%
\left\Vert \widehat{f}_{n}\circ \phi _{n}-j\circ f_{n}\right\Vert _{k}\leq
100\dim \left( X_{n}\right) \delta ^{\frac{1}{2}}$. Moreover if $Z$ is an $%
M_{q}$-system, $g_{n}:\widehat{X}_{n}\rightarrow Z$ and $i:Y\rightarrow Z$
are ucp maps such that $\left\Vert g_{n}\circ \phi _{n}-i\circ
f_{n}\right\Vert _{k}\leq 100\dim \left( X_{n}\right) \delta _{n}^{\frac{1}{2%
}}$, then there exists a unique linear map $\tau :\widehat{Y}\rightarrow Z$
such that $\tau \circ j=i$ and $\tau \circ \widehat{f}_{n}=g_{n}$.
Furthermore $\tau $ is ucp.
\end{lemma}

The following characterization of nuclearity for operator systems can then
be proved similarly as \ref{Lemma:ac-sp}. Again one needs to use \cite[Lemma
3.3]{lupini_universal_2014}.

\begin{lemma}
\label{Proposition:ac-sy}Suppose that $Z$ is a $1$-exact operator system.
The following statements are equivalent:

\begin{enumerate}
\item For every finite-dimensional operator systems $E,F$, $\delta >0$,
unital linear maps $\phi :E\rightarrow F$ and $f:E\rightarrow Z$ such that $%
\phi $ is injective and $\left\Vert f\circ \phi ^{-1}\right\Vert
_{cb}<1+\delta $, there is a ucp map $g:F\rightarrow Z$ such that $%
\left\Vert g_{|E}-f\right\Vert <100\dim \left( E\right) \delta ^{\frac{1}{2}%
} $;

\item For every $q\in \mathbb{N}$, $E\subset M_{q}$ subsystem, ucp map $%
f:E\rightarrow Z$, and $\varepsilon >0$, there exits a ucp map $%
g:F\rightarrow Z$ such that $\left\Vert g\circ \phi -f\right\Vert
<\varepsilon $;

\item $Z$ is nuclear.
\end{enumerate}
\end{lemma}

The operator system analog of Kirchberg's nuclear embedding theorem can also
be obtained with similar methods as Theorem \ref{Theorem:embedding-sp}.

\begin{theorem}
\label{Theorem:embedding-sy}Suppose that $X$ is a separable $1$-exact
operator system. Then $X$ is nuclear if and only if there exist a complete
order embedding $\phi :X\rightarrow \mathbb{GS}$ and a unital completely
positive projection of $\mathbb{GS}$ onto the range of $\phi $.
\end{theorem}

In order to prove Theorem \ref{Theorem:embedding-sy} one can proceed as in
the proof of Theorem \ref{Theorem:embedding-sp}, by replacing \cite[%
Proposition 4.8]{lupini_uniqueness_2014} with the characterization of $%
\mathbb{GS}$ given by \cite[Proposition 4.2]{lupini_universal_2014}. All
linear maps should be replaced by unital linear maps. In this setting the
functor $\mathrm{OM\mathrm{IN}}_{q}$ plays the role of $\mathrm{MIN}_{q}$.

\bibliographystyle{amsplain}
\bibliography{bib-ac}

\end{document}